\documentclass{amsart}
\usepackage[cp1251]{inputenc}
\usepackage[english,russian]{babel}
\usepackage{amsmath}
\usepackage{amssymb}
\usepackage{amsfonts}
\def\udcs{517.16 + 517.574 + 517.555}

\newtheorem{proposition}{Утверждение}
\newtheorem{lemma}{Лемма}

\newcommand{\R}{\mathbb{R}}
\newcommand{\Z}{\mathbb{Z}}
\newcommand{\bC}{\mathbb{C}}
\newcommand{\N}{\mathbb{N}}

\renewcommand{\leq }{\leqslant }
\renewcommand{\geq }{\geqslant }

\renewcommand{\d}{{\rm \, d}}
\DeclareMathOperator{\Inc}{Inc}

\newtheorem{Conjec}{Гипотеза}

\begin{document}
УДК \udcs
\thispagestyle{empty}

\title[Три эквивалентные гипотезы о некоторых оценках интегралов\dots]{Три эквивалентные гипотезы\\ об одной оценке интегралов}

\author{\large Р.\,А. Баладай, Б.\,Н. Хабибуллин}

\address{Рустам Алексеевич Баладай,
\newline\hphantom{iii} Башкирский государственный университет,
\newline\hphantom{iii} ул. З.~Валиди, 32, 
\newline\hphantom{iii} 450074, г. Уфа, Россия}
\email{baladaichik@mail.ru}
 
\address{Булат Нурмиевич Хабибуллин, 
\newline\hphantom{iii} Башкирский государственный университет,
\newline\hphantom{iii} ул. З. Валиди, 32, 
\newline\hphantom{iii} 450074, г. Уфа, Россия}
\email{Khabib-Bulat@mail.ru}

\thanks{\sc R.\,A.Baladai, B.\,N. Khabibullin, 
Three equivalent conjectures on an estimate of integrals.}
\thanks{\copyright \ Баладай Р.,А., Хабибуллин Б.\,Н. 2010}
\thanks{\rm Работа поддержана РФФИ (гранты 09-01-00046-а, 08-01-97023-р\_поволжье\_а) и программой ``Государственная поддержка ведущих научных школ'', проект НШ--3081.2008.1.}
\thanks{\it Поступила 15 июня  2010 г.}

\maketitle {\small
\begin{quote}
\noindent{\bf Аннотация. } В работе выдвигается гипотеза о точной оценке некоторого определенного несобственного интеграла, зависящего от параметра $\lambda \in (0,+\infty)$,  через заданную оценку другого определенного интеграла, зависящего от двух параметров $t\in [0,+\infty)$ и $\lambda$. Такая точная оценка доказана здесь для $\lambda\leq 1$. Кроме того, получена некоторая оценка и при $\lambda>1$. Последняя оценка, по-видимому, не точная. Мы приводим также две гипотезы, эквивалентные исходной. Истоки наших гипотез --- экстремальные задачи для целых, мероморфных и плюрисубгармонических функций нескольких переменных. 
\medskip

\noindent{\bf Ключевые слова:}{несобственный интеграл, оценка, неравенство, целая функция, мероморфная функция, плюрисубгармоническая фу\-н\-к\-ция, проблема Пэли}
\medskip
\end{quote}
\begin{quote}
\noindent{\bf Abstract. } We offer a conjecture on sharp estimation of a definite improper integral depend on a parameter $\lambda \in (0,+\infty)$ by means of given estimate of other definite integral depend on parameters $t\in [0,+\infty)$ and $\lambda$. Such sharp estimate is proved for $\lambda \leq 1$. Besides, an estimate is obtained for $\lambda >1$. The last estimate is not exact seemingly. We give also two conjectures that are equivalent to the original conjecture. Sources of our conjectures are extremal problems for entire, meromorphic, and plurisubharmonic functions of several variables.
\medskip

\end{quote} }

\section{Введение}

<<{\it Положительность\/}>> (<<{\it отрицательность\/}>>) всюду в работе означает ``$\geq 0$\,'' (соотв.\footnote{Далее сокращение <<соотв.>> используется для наречия или предлога <<соответственно>>.} ``$\leq 0$\,''), где отношение порядка $\leq$ и нулевой элемент $0$ на рассматриваемом множестве, как правило, естественны по контексту.

Через $\N$, $\Z$, $\R$, $\bC$ обозначаем множества всех соотв. {\it натуральных}, {\it целых}, {\it вещественных\/} и {\it комплексных\/}  чисел; $[-\infty, +\infty]:=\{-\infty\}\cup \R \cup \{+\infty\}$   ---  {\it расширенная вещественная ось\/} с естественным отношением порядка. 
 
Функция $\phi \colon I\to [-\infty, +\infty]$, $I\subset [-\infty , +\infty]$, {\it возрастающая\/} (соотв. {\it убывающая\/}), если для любых $x_1, x_2 \in I$ неравенство $x_1\leq x_2$ влечет за собой {\it нестрогое неравенство\/} $\phi (x_1)\leq \phi (x_2)$ (соотв. $\phi (x_1)\geq \phi (x_2)$). Если же для любых $x_1,x_2 \in I$ из {\it строгого неравенства\/} $x_1<x_2$  следует {\it строгое неравенство\/} $\phi (x_1)<\phi (x_2)$ (соотв. $\phi(x_1)>\phi(x_2)$), то $\phi$ --- {\it строго возрастающая\/} (соотв. {\it строго убывающая\/}) функция на $I$. 
Через $\Inc (I)$ обозначаем множество всех возрастающих функций на $I$. Верхний индекс ``$^+$'', наряду с тем, что обозначает положительную часть $a^+:=\max \{a,0\}$ (соотв. $\phi^+:=\sup\{\phi, 0\}$) числа $a$ или функции $\phi \colon \cdot \to [-\infty , +\infty]$, используется также для обозначения всех положительных элементов из множества или класса функций. Так, подмножество всех положительных функций из $\Inc (I)$ обозначаем $\Inc^+(I)$.

Функция $\phi \colon [a,b)\to \R$, $[a,b)\subset [0,+\infty]$, {\it выпуклая относительно\footnote{Если не указано основание логарифма у $\log$, то оно равно числу $e$, т.\,е. это функция $\ln$.}\/} $\log$, если функция $x\mapsto \phi(e^x)$ выпуклая на интервале $[\log a, \log b)\subset [-\infty , +\infty]$.    

В сязи с исследованием в работе \cite[Теоремы 1, 2]{Kh99} экстремальных задач о росте плюрисубгармонических, целых и мероморфных функций  комплексных переменных в завершение обзора \cite[Commentary]{Kh02} была выдвинута     
\begin{Conjec}\label{con:1} Пусть функция $S\in \Inc^+\bigl([0, +\infty)\bigr)$  выпуклая относительно\/ $\log$ и для $\lambda \geq 1/2$ и $2\leq n\in \N$ выполнено неравенство
\begin{equation}\label{conj:con}
\int_0^1S(tx)(1-x^2)^{n-2} x\d x\leq t^{\lambda} \; \text{ для всех \; $t\in [0,+\infty)$}.
\end{equation}
Тогда
\begin{equation}\label{conj:est}
\int_0^{+\infty} S(t)\frac{t^{2\lambda-1}}{(1+t^{2\lambda})^2} \d t\leq  \frac{\pi (n-1)}{2\lambda}\prod_{k=1}^{n-1}\Bigl(1+\frac{\lambda}{2k}\Bigr).
\end{equation}
\end{Conjec}
Если Гипотеза \ref{con:1} верна, то в точности оценки \eqref{conj:est} при условии 
\eqref{conj:con} легко убедиться на примере возрастающей выпуклой относительно $\log$ функции
\begin{equation*}
S_{\lambda,n}(t):=2(n-1)\prod_{k=1}^{n-1} \Bigl(1+\frac{\lambda}{2k}\Bigr)\, t^{\lambda}=:c_{\lambda,n}t^{\lambda}, \quad  t\in [0, +\infty),\quad \lambda \geq\frac{1}{2}	\, ,
\end{equation*}
где постоянная $c_{\lambda,n}$ определена последним равенством.
Для $S_{\lambda, n}$ в  \eqref{conj:con} и \eqref{conj:est} достигается равенство. Действительно, в обозначениях $\Gamma$ и $\mathrm B$ для классических гамма- и бета-функций Эйлера \cite[гл.~VII, \S~1]{LSh} получаем (см. \eqref{conj:con})
\begin{multline}\label{with:B}
\int_0^1	S_{\lambda,n}(tx)(1-x^2)^{n-2} x\d x =\frac{c_{\lambda, n}}{2}\, t^{\lambda}
\int_0^1	x^{(\lambda/2+1)-1}(1-x)^{(n-1)-1} \d x \\=\frac{c_{\lambda, n}}{2}\, t^{\lambda} \mathrm B (\lambda/2+1,n-1) 
=\frac{c_{\lambda, n}}{2}\, t^{\lambda}\frac{\Gamma(\lambda/2
+1)\ \Gamma(n-1)}{\Gamma(\lambda/2+n)}\\
=\frac{c_{\lambda, n}}{2}\, t^{\lambda} \frac{(\lambda/2)\Gamma(\lambda/2)\ (n-2)!}{\Gamma(\lambda/2)\ (\lambda/2)\ (\lambda/2+1)\cdot\ldots\cdot(\lambda/2+(n-1))}\\
=c_{\lambda,n}\,\dfrac{1}
{2(n-1)\prod\limits^{n-1}_{k=1}\Bigl(1+\dfrac{\lambda}{2\,k}
\Bigr)}\, t^{\lambda}= t^{\lambda},
\end{multline}
а также  (см. \eqref{conj:est})
\begin{multline}\label{eq:Sl}
\int_0^{+\infty} S_{\lambda,n}(t)\frac{t^{2\lambda-1}}{(1+t^{2\lambda})^2} \d t=
-\frac{c_{\lambda, n}}{2\lambda}\int_0^{+\infty}t^{\lambda}\d \frac{1}{1+t^{2\lambda}} \\
=\frac{c_{\lambda, n}}{2\lambda}\int_0^{+\infty}\frac{1}{1+t^{2\lambda}}\d t^{\lambda}=\frac{c_{\lambda, n}}{2\lambda}\int_0^{+\infty}\frac{1}{1+s^2}\d s =
\frac{\pi (n-1)}{2\lambda}\prod_{k=1}^{n-1}\Bigl(1+\frac{\lambda}{2k}\Bigr).
\end{multline}

\section{Случай лишь положительной возрастающей функции $S$}

В этом разделе мы получим некоторые оценки для интеграла \eqref{conj:est} от функции $S$ \underline{без каких-либо условий выпуклости}.  Так, при $\lambda \leq 1$ справедлив следующий более общий вариант Гипотезы \ref{con:1}.
\begin{proposition}  Пусть $S\in \Inc^+\bigl([0, +\infty)\bigr)$ --- непрерывная функция и для 
\begin{equation*}
	\frac{1}2\leq \lambda \leq 1
\end{equation*}
 и $2\leq n\in \N$ выполнено \eqref{conj:con}. Тогда имеет место неулучшаемая оценка \eqref{conj:est}.
\end{proposition}
\begin{proof} Сначала дадим грубую оценку  функции $S$ при \underline{произвольном} значении параметра $\lambda$. Такая оценка, в частности, необходима нам для сходимости возникающих в ходе  доказательства различных интегралов и для существования некоторых пределов. Применения этой Леммы \ref{lem:1} далее не объявляются. 
\begin{lemma}\label{lem:1} Пусть $S\in \Inc^+\bigl([0, +\infty)\bigr)$, $\lambda \geq 0$ и для 
$2\leq n\in \N$ выполнено \eqref{conj:con}. Тогда
\begin{equation*}\label{est:S}
S(x)\leq 2(n-1)\left(1+\frac{\lambda}{2(n-1)}\right)^{n-1}
\Bigl(1+\frac{2(n-1)}{\lambda}\Bigr)^{\lambda/2} x^\lambda
\quad \text{ при  всех $x\geq 0$}.	
\end{equation*}
\end{lemma}
\begin{proof}[Доказательство леммы\/ {\rm \ref{lem:1}}] Для произвольного числа $a\in [0,1]$ в силу возрастания $S$ имеем
\begin{multline*}
S(at)\leq \int_a^1S(tx)(1-x^2)^{n-2} x\d x\biggm/
\int_a^1(1-x^2)^{n-2} x\d x\\\stackrel{\eqref{conj:con}}{\leq}
t^{\lambda}\biggm/\int_a^1(1-x^2)^{n-2} x\d x =\frac{2(n-1)}{(1-a^2)^{n-1}}\, t^{\lambda}=
2(n-1)\frac{1}{a^{\lambda}(1-a^2)^{n-1}}\, (at)^{\lambda}.
 \end{multline*}
Используя замену $x=at$, а затем минимизируя по $a\in [0,1]$ дробь в правой части, получаем требуемое  \eqref{est:S}.  
\end{proof}

Далее нам удобнее перейти к новым переменным: обозначить переменную $t$ через $r$, $rx$ заменить на $t$, а вместо  функции $S$ рассмотреть функцию 
\begin{equation*}
T:=\dfrac{1}{2(n-1)}\, S.	
\end{equation*}

Тогда неравенство \eqref{conj:con} записывается в виде 
\begin{equation}\label{est:an} 
2(n-1)\int_{0}^{r}{T(t)(r^2-t^2)^{n-2}t\d t \leq r^{\lambda+2(n-1)}}
\quad \text{\it для любых\/ $r\geq 0$},
\end{equation}
а неравенство \eqref{conj:est} перейдет в требующее доказательства неравенство
\begin{equation}\label{est:inf}
\int_0^{+\infty}T(t)\frac{t^{2\lambda-1}}{(1+t^{2\lambda})^{2}}
\d t \, {\stackrel{\text{?}}{\leq}} 
\, \frac{\pi}{4\lambda}\cdot  
\prod_{k=1}^{n-1}\Bigl(1+\frac{\lambda}{2k}\Bigr).
\end{equation}
 Для непрерывной функции $f$ на $[0,+\infty)$ введем интегральный оператор
\begin{equation}\label{df:Ikf} 
I_k(r;f):=\int_{0}^{r}f(t)(r^2-t^2)^ktdt, \quad r\geq 0, \quad k\in \N.
\end{equation}
В частности, интеграл из \eqref{est:an} есть в точности $I_{n-2}(r;T)$.

Далее, введем в рассмотрение операторы $L$ и $M$, действующие на дифференцируемые функции $g\colon (0,+\infty) \to \R$ по правилу
\begin{equation}\label{df:LMr} 
L[g](r):=\frac{1}{r}\cdot g'(r), \quad  M[g](r):=
-\frac{d}{dt} \left(\frac{1}{t}\cdot g(t)\right)\biggm|_{t=r}\, , \quad r>0.  
\end{equation}
Легко показать, что 
\begin{align}
L^p[I_k(\cdot ;f)](r)&=2^p \frac{k!}{(k-p)!}\, I_{k-p}(r;f), \quad p\leq k, \label{Lkp3}\\ 
L^{k+1}[I_k(\cdot ;f)](r)&=2^k k!\, f(r), \label{Lkp4}
\end{align}
где $p$-ая степень оператора означает его $k$-кратную суперпозицию.
\begin{lemma}\label{lem:2} 
Предположим, что функции $g$ и $\varphi$ соотв.\/ $p$ раз и $q+1$ раз непрерывно дифференцируемы на $(0,+\infty)$ и существуют пределы 
\begin{equation*}
  \lim_{t\to 0,+\infty} \frac{1}{t}\, L^{p-1}[g](t)\cdot M^q[\varphi](t)=0. 
\end{equation*} 
Тогда
\begin{equation*}
 \int_{0}^{+\infty}L^p[g](t) \, M^q[\varphi](t) \d t=\int_{0}^{+\infty}L^{p-1}[g](t)\, M^{q+1}[\varphi](t)\d t, 
\end{equation*}
если интегралы сходятся.
\end{lemma}

Лемму \ref{lem:2} легко доказать интегрированием по частям.

\begin{lemma}\label{lem:3} Пусть 
\begin{equation}\label{df:varl} \varphi_\lambda(t):=\frac{t^{2\lambda-1}}{(1+t^{2\lambda})^2},\quad  t\geq 0. 
\end{equation}

Тогда для всех\/ $q=0,1, \dots$ имеют место соотношения
\begin{align}
 M^q[\varphi_\lambda](t)&=O(t^{-2\lambda-1-2q}) \quad 
 \text{ при }\; t\to +\infty ,& \label{O:inf}\\ 
M^q[\varphi_\lambda](t)&=O(t^{2\lambda-1-2q}) \; \; \quad 
\text{ при }\;  t\to 0 .& \label{O:0}
\end{align}
 Кроме того, при\/ $\lambda \leq 1$ выполнено условие положительности
 \begin{equation}\label{M:pos} 
M^q[\varphi_\lambda](t)\geq 0 \quad \text{для всех }\; t\geq 0.
\end{equation}
\end{lemma}
\begin{proof} Соотношения\/ \eqref{O:inf}  и \eqref{O:0} можно получить, предста\-в\-л\-яя фу\-н\-к\-ц\-ию $\varphi_\lambda$  в виде ряда по степеням от соотв. $t^{-\lambda}$ при $t\to +\infty$ и  $t^{\lambda}$ при $t\to 0$.
  
Для доказательства  \eqref{M:pos} при $\lambda \leq 1$ введем
для $0\leq\ \beta\leq2$   класс функций $K(\beta)$,  являющихся линейными комбинациями с положительными коэффицентами функций вида
\begin{equation*} 
\psi(t)=\frac{1}{t^\alpha(1+t^\beta)^k}\, ,\quad  \alpha \geq -1, 
\; k\geq 0. 
\end{equation*}

Например, для нашей функции из \eqref{df:varl} имеем $\varphi_\lambda\in K(2\lambda)$ при $\lambda\leq 1$.
Поэтому для получения \eqref{M:pos} достаточно показать, что $M[\psi] \in K(\beta)$ для $\psi\in K(\beta)$. Для $0\leq\beta\leq2$ и $\alpha\geq-1$ получаем
 \begin{equation*} 
 M[\psi](t)=-\frac{d}{dt} \frac{1}{t^{1+\alpha}{(1+t^\beta)}^k}= \frac{1+\alpha}{t^{2+\alpha}{(1+t^\beta)}^k} + \frac{k\beta}{t^{2-\beta+\alpha}{(1+t^\beta)}^{k+1}}\in K(\beta).  
\end{equation*} 
 Лемма \ref{lem:3} доказана.
 \end{proof}
 \begin{lemma}\label{lem:4} Для всех $k=0,1,\dots$ и $p=0,1,\dots , k+1$
 \begin{equation}\label{Lpk} L^p[I_k(\cdot;T)](r)=O\left(r^{\lambda+2(k+1-p)}\right),  \quad r\rightarrow 0,+\infty. 
 \end{equation}
 \end{lemma}
 \begin{proof} Для $p=k+1$ из равенства \eqref{Lkp4}  имеем
  \begin{equation*} 
  L^{k+1}[I_k(\cdot;T)](r)=2^k\cdot k!\,T(r)
  =O(r^\lambda)
  =O\bigl(r^{\lambda+2(k+1-p)}\bigr), \quad 
   r\rightarrow 0,+\infty. 
   \end{equation*}  
   Для $p\leq k$ из равенства \eqref{Lkp3} следует
\begin{multline*} 
   L^p[I_k(\cdot ;T)](t) 
  =O\bigl(I_{k-p}(r;T)\bigr)
   = O\Bigl(\,\int_0^r t^{\lambda+1}(r^2-t^2)^{k-p} \d t \, \Bigr) 
  \\
   = O\Bigl(\,r^{2(k-p)}\int_0^r t^{\lambda+1}\d t\,\Bigr)
    = O\bigl(r^{\lambda+2(k+1-p)}\bigr),\quad  r\rightarrow 0,+\infty.  
  \end{multline*}
  Лемма \ref{lem:4} доказана.
  \end{proof}

  \begin{lemma}\label{lem:5} Для $p+q=n-1$, где $p\geq 0$ и $q\geq 0$,
  имеем 
  \begin{equation*} 
  \lim_{t\rightarrow 0,+\infty}\frac{1}{t}\, L^{p-1}[I_{n-2}(\cdot,T)](t)\, M^q[\varphi_\lambda](t)=0. 
  \end{equation*}
\end{lemma}  
   \begin{proof}
   По Леммам \ref{lem:3} и \ref{lem:4} при $t\to 0$ получаем
  \begin{equation*} \frac{1}{t}\, L^{p-1}[I_{n-2}(\cdot ,T)](t)\, M^q[\varphi_\lambda](t)=\frac{1}{t}\, O\bigl(t^{\lambda+2(n-p)}\bigr)\cdot O\bigl(t^{2\lambda-1-2q}\bigr)=O(t^{3\lambda}). \end{equation*} 
   По тем же Леммам \ref{lem:3} и \ref{lem:4} при $t\to +\infty$ имеем
 \begin{equation*} \frac{1}{t}\, L^{p-1}[I_{n-2}(\cdot,T)](t)\, M^q[\varphi_\lambda](t)=\frac{1}{t}\, O\bigl(t^{\lambda+2(n-p)}\bigr)\cdot O\bigl(t^{-2\lambda-1-2q}\bigr)=O(t^{-\lambda}). \end{equation*}
 Лемма \ref{lem:5} доказана.
 \end{proof}
 Закончим доказательство Гипотезы \ref{con:1} для $\lambda\leq 1$.
 По \eqref{Lkp4} левая часть \eqref{est:inf} равна
 \begin{multline*}
J\stackrel{\eqref{df:varl}}{:=}\int_0^{+\infty} T(t)\varphi_\lambda(t)\d t\\
\stackrel{\eqref{df:LMr}-\eqref{df:Ikf}}{=}\frac{1}{2^{n-2}\cdot(n-2)!}\int_0^{+\infty}{L^{n-1}[I_{n-2}(\cdot,T)](t)\,M^0[\varphi_\lambda](t)\d t} 
\end{multline*}
 Используем $n-1$ раз Лемму \ref{lem:2}  и получим
 \begin{multline*}  J=\frac{1}{2^{n-2}\cdot(n-2)!} \int_0^{+\infty} L^0[I_{n-2}(\cdot,T)](t)\,M^{n-1}[\varphi_\lambda](t)\d t \\  = \frac{1}{2^{n-2}\cdot(n-2)!} \int_{0}^{+\infty}{\left(\int_{0}^{t}T(\tau)(t^2-\tau^2)^{n-2}\tau \d \tau\right)}M^{n-1}[\varphi_\lambda](t)\d t. 
 \end{multline*}
 Отсюда по  соотношение \eqref{M:pos}  Леммы \ref{lem:2} из условия \eqref{est:an} ввиду положительности функции $M^{n-1}[\varphi_\lambda]$ для $\lambda\leq 1$  получаем
 \begin{multline*} 
 J\leq \frac{1}{2^{n-1}\cdot(n-1)!}
\int_0^{+\infty} t^{\lambda+2(n-1)}\, M^{n-1}[\varphi_\lambda](t) \d t \\  =\frac{1}{2^{n-1}\cdot(n-1)!}\int_0^{+\infty}L^0\bigl[t^{\lambda+2(n-1)}\bigr]\, M^{n-1}[\varphi_\lambda](t) \d t.  \end{multline*}
 Вновь используя $n-1$ раз Лемму \ref{lem:2},  устанавливаем оценку
 \begin{multline*}
  J\leq \frac{1}{2^{n-1}\cdot(n-1)!}
\int_0^{+\infty}L^{n-1}\bigl[t^{\lambda+2(n-1)}\bigr]\varphi_\lambda(t)\d t  \\  =\frac{1}{2^{n-1}\cdot(n-1)!}\int_0^{+\infty}\frac{t^{3\lambda-1}}{(1+t^{2\lambda})^2}dt \cdot \prod_{k=1}^{n-1}{(\lambda+2k)}\\
=\frac{1}{2\lambda}\int_{0}^{+\infty}{\frac{t^{1/2}}{(1+t)^2}dt\cdot \prod_{k=1}^{n-1}{\Bigl(1+\frac{\lambda}{2k}\Bigr)}}. \end{multline*}
 Отсюда, интегрируя по  частям последний интеграл, получаем 
\begin{equation*} 
 J\leq \frac{1}{4\lambda}\int_{0}^{+\infty}{\frac{t^{-1/2}}{(1+t)}\d t \cdot \prod_{k=1}^{n-1}{\Bigl(1+\frac{\lambda}{2k}\Bigr)}}. 
 \end{equation*}
 Интеграл здесь легко (после Л.~Эйлера) вычисляется с помощью вычетов и равен  $\pi$ 
 \cite[гл.~V, 74, Пример~3]{LSh}. Значит  последнее неравенство совпадает с неравенством \eqref{est:inf}.  Это завершает доказательство Гипотезы \ref{con:1} для $\lambda\leq 1$.
\end{proof}

\section{Эквивалентные версии Гипотезы \ref{con:1}}

\begin{lemma}[{\cite[Предложение 5.1]{Kond}}]\label{lem:6} Функция $S \colon [0, +\infty) \to \R$ с $S(0)=0$ --- возрастающая и выпуклая относительно  $\log$, если и только если найдется возрастающая функция $s\colon [0, +\infty) \to [0, +\infty)$, дающая представление
\begin{equation*}
	S(x)=\int_0^x \frac{s(t)}{t} \d t.
\end{equation*}
\end{lemma}
Используя Лемму \ref{lem:6} вместе с Леммой \ref{lem:1} для условия $S(0)=0$, можем записать интеграл из \eqref{conj:con} в виде
\begin{equation*}
\int_0^1 S(tx)(1-x^2)^{n-2}x \d x =-\frac1{2(n-1)} 
\int_0^1 \left(\int_0^{tx}\frac{s(\tau)}{\tau} \d \tau\right) \d (1-x^2)^{n-1}.
\end{equation*}
Интегрирование по частям дает равенство
\begin{equation*}
\int_0^1 S(tx)(1-x^2)^{n-2}x \d x =\frac1{2(n-1)} 
\int_0^1 \frac{s(tx)}{x} (1-x^2)^{n-1} \d x.
\end{equation*}
Аналогично для интеграла \eqref{conj:est} имеем
\begin{equation*}
	\int_0^{+\infty} 	S(t)\frac{t^{2\lambda-1}}{(1+t^{2\lambda})^2} \d t
= \frac1{2\lambda} \int_0^{+\infty} 	\frac{s(t)}{t}\,\frac{\d t}{t(1+t^{2\lambda})} \,.
\end{equation*}
Таким образом, соотношения \eqref{conj:con} и \eqref{conj:est} переходят соотв. в неравенства
\begin{subequations}
\begin{align*}
\frac1{2(n-1)} 
\int_0^1 \frac{s(tx)}{x} (1-x^2)^{n-1} \d x &\leq t^{\lambda}\, 
\; \text{ для всех } \; 0\leq t < +\infty \,,  \\ 
\frac1{2\lambda} \int_0^{+\infty} 	\frac{s(t)}{t}\,\frac{\d t}{1+t^{2\lambda}}\leq \frac{\pi (n-1)}{2\lambda}\prod_{k=1}^{n-1} &\Bigl(1+\frac{\lambda}{2k}\Bigr)= \frac{\pi
(n-1)}{\lambda^2} \cdot \frac{1}{\mathrm B (\lambda/2, n)} \, ,
\end{align*}
\end{subequations}
где функция $s \geq 0$ --- возрастающая, а в последнем равенстве использованы те же свойства бета-функции Эйлера \cite[гл.~VII, \S~1]{LSh},  что и в \eqref{with:B}. Но мы не уменьшим общности рассмотрения, если вместо такой произвольной функции $s$ будем рассматривать произвольную возрастающую функцию $h\geq 0$, определенную заменой $h(x^2):=\frac1{4(n-1)}\,s(x)$, $x\geq 0$, что преобразует  последнюю пару соотношений в  
условие и неравенство
\begin{subequations}
\begin{align*}
2 \int_0^1 \frac{h(t^2x^2)}{x} (1-x^2)^{n-1} \d x &\leq (t^2)^{\lambda/2}\, 
\; \text{ для всех }
\; 0\leq t < +\infty \,,  
\\ 
 2\int_0^{+\infty} 	\frac{h(t^2)}{t}\,\frac{\d t}{1+t^{2\lambda}}&\leq
\frac{\pi}{2} \prod_{k=1}^{n-1} \Bigl(1+\frac{\lambda}{2k}\Bigr)
=\frac{\pi}{\lambda}\cdot \frac{1}{\mathrm B (\lambda/2, n)} 	\, ,
\end{align*}
\end{subequations}
Отсюда после замен 
\begin{equation*}
x^2=x', \quad t^2=t', \quad \lambda/2=\alpha>1/2	
\end{equation*}
и переобозначения переменных $x'$ и $t'$ прежними $x$ и $t$ получаем
\begin{subequations}\label{in:sh2}
\begin{align}
 \int_0^1 \frac{h(tx)}{x} \,(1-x)^{n-1} \d x &\leq t^{\alpha}\, 
\; \text{ для всех }
\; 0\leq t < +\infty \,,  \label{in:sh21}\\ 
 \int_0^{+\infty} 	\frac{h(t)}{t}\,\frac{\d t}{1+t^{2\alpha}}&\leq
\frac{\pi}{2} \prod_{k=1}^{n-1} \Bigl(1+\frac{\alpha}{k}\Bigr)=
\frac{\pi}{2\alpha} \cdot \frac{1}{\mathrm B (\alpha, n)}\, .
\label{in:sh2i}
\end{align}
\end{subequations}
Таким образом, Гипотезе \ref{con:1} с $\lambda >1$ эквивалентна более простая
\begin{Conjec}\label{Hyp2} Пусть $\alpha >1/2$. Для любой возрастающей функции $h\geq 0$ на $[0,+\infty)$ из условия \eqref{in:sh21} следует неравенство \eqref{in:sh2i}.
\end{Conjec}

Возможна еще одна версия рассмотренных гипотез. Во-первых отметим, что из  некоторых соображений  достаточно доказывать наши гипотезы для гладких функций. Так, можем предполагать, что функция $h$ в Гипотезе \ref{Hyp2}
непрерывно дифференцируема, т.\,е. $q:=h'\geq 0$ на $(0,+\infty)$. Тогда интегрированием по частям
получается эквивалентная Гипотезам \ref{con:1} и \ref{Hyp2}
\begin{Conjec} Пусть $\alpha >1/2$. Если $q$ --- положительная непрерывная функция на $[0,+\infty)$, то из
условия
\begin{equation}\label{verq:i}
\int_0^1 \Bigl(\,\int_x^1 (1-y)^{n-1} \frac{\d y}{y}\Bigr)q(tx)\d x
\leq t^{\alpha-1}\, \; \text{ для всех }
\; 0\leq t < +\infty	
\end{equation}
следует оценка
\begin{equation*}
	 \int_0^{+\infty} q(t)\log \Bigl(1+\frac1{t^{2\alpha}}\Bigr) \d t\leq
\pi \alpha \prod_{k=1}^{n-1} \Bigl(1+\frac{\alpha}{k}\Bigr)=
 \frac{\pi}{\mathrm B (\alpha, n)}\,.
\end{equation*}
\end{Conjec}
Внутренний интеграл в левой части \eqref{verq:i} легко вычисляется и может быть заменен на сумму:
\begin{equation*}
\int_x^1 (1-y)^{n-1} \frac{\d y}{y}=
-\log x +\sum_{k=1}^{n-1}\frac{(-1)^k}{k}\,(1-x^k).	
\end{equation*}

\section{Некоторые оценки интеграла из \eqref{conj:est} при условии \eqref{conj:con}}

В этом последнем разделе мы приведем некоторые оценки, которые не ``дотягивают'' до Гипотезы \ref{con:1}, но представляют некоторый интерес. Первая из них сразу следует из заключения eqref{est:S} Леммы \ref{lem:1} и цепочки равенств \eqref{eq:Sl}.
\begin{proposition} Пусть $S\in \Inc^+\bigl([0,+\infty)\bigr)$, $\lambda \geq 0$ и для 
$2\leq n\in \N$ выполнено условие \eqref{conj:con}. Тогда справедлива оценка
\begin{equation}\label{est:S0}
\int_0^{+\infty} S(t)\frac{t^{2\lambda-1}}{(1+t^{2\lambda})^2} \d t\leq
\frac{\pi}{2}\cdot
\frac{n-1}{\lambda}\left(1+\frac12
\cdot \frac{\lambda}{n-1}\right)^{n-1}
\Bigl(1+2\cdot \frac{n-1}{\lambda}\Bigr)^{\lambda/2} .	
\end{equation}
\end{proposition}

По схеме доказательства последнего неравенства в \cite[Основная лемма]{Kh99} может быть установлено
\begin{proposition} В условиях Гипотезы\/ {\rm \ref{con:1}} справедлива оценка
\begin{equation}\label{conj:est2}
\int_0^{+\infty} S(t)\frac{t^{2\lambda-1}}{(1+t^{2\lambda})^2} \d t\leq  \frac{\pi (n-1)}{2\lambda}\prod_{k=1}^{n-1} b\Bigl(\frac{\lambda}{2k}\Bigr),
\end{equation}
где функция $b\colon [0,+\infty)\to \R$ определена по правилу

\begin{equation*}
b(x):=\begin{cases}
e^x\quad &\text{ при }\; x\leq 1,\\
ex\quad &\text{ при }\; x>1,
\end{cases}	
\end{equation*}
и удовлетворяет неравенству $b(x)\leq e(1+x)$ при всех \; $x\geq 0$.
\end{proposition}

Обсуждения Гипотезы \ref{con:1} приведены также в 
\cite{Kh_arXiv} и \cite{Kh_uns}.

\end{document}